\documentclass[10pt]{article}
\usepackage{amsthm, amsmath, amssymb,enumerate}
\usepackage{fullpage}
\pdfoutput=1

\newtheorem{thm}{Theorem}[section]
\newtheorem{cor}[thm]{Corollary}
\newtheorem{lem}[thm]{Lemma}
\newtheorem*{lem:tracelemma}{Lemma \ref{tracelemma}}
\newtheorem*{lem:techlem2}{Lemma \ref{techlem2}}
\newtheorem{prop}[thm]{Proposition}
\newtheorem*{mainsystolethm}{Theorem \ref{mainsystole}}
\newtheorem*{linkthm}{Theorem \ref{thm1}}
\usepackage[pdftex]{graphicx,color}
\newcommand{\chat}{\hat{\mathbb{C}}}
\newcommand{\HH}{\mathbb{H}}
\newcommand{\bQ}{\mathbb{Q}}
\newcommand{\bZ}{\mathbb{Z}}

\newcommand{\pslc}{\mathrm{PSL}_2(\mathbb{C})}
\newcommand{\slc}{\mathrm{SL}_2(\mathbb{C})}
\newcommand{\tr}{\mathrm{tr} \,}
\newcommand{\arccosh}{\operatorname{arccosh}}
\newcommand{\vol}{\mathrm{vol}}
\newcommand{\sys}{\mathrm{sys}}

\newcommand{\psl}{\mathrm{PSL}}

\title{Systoles and Dehn surgery for hyperbolic $3$--manifolds.}
\author{Grant S. Lakeland and Christopher J. Leininger}

\begin{document}

\maketitle

\begin{abstract}
Given a closed hyperbolic 3-manifold $M$ of volume $V$, and a link $L \subset M$ such that the complement $M \setminus L$ is hyperbolic, we establish a bound for the systole length of $M \setminus L$ in terms of $V$. This extends a result of Adams and Reid, who showed that in the case that $M$ is not hyperbolic, there is a universal bound of $7.35534...$ . As part of the proof, we establish a bound for the systole length of a non-compact finite volume hyperbolic manifold which grows asymptotically like $\frac{4}{3} \log{V}$.
\end{abstract}

\section{Introduction}

Given a Riemannian $n$--manifold $M$ the \emph{systole} of $M$ is defined to be the shortest closed geodesic on $M$, and the \emph{systole length} $\sys(M)$ provides an interesting invariant of $M$. Gromov \cite{GromovFilling} proved that in each dimension $n$, there is a universal constant $C_n$ such that
\[ \sys(M)^n \leq C_n \vol(M)\]
(see also Guth \cite{Guth}). For an overview of this and other connections to geometric and topological invariants see \cite{KatzBook}.

For a hyperbolic manifold $M$, there is a bound on $\sys(M)$ in terms of $\vol(M)$ which is logarithmic; for example, see \cite{AdamsSystole,Parlier,Schmutz2,Schmutz1} in dimension two, and \cite{Calegari,Gendulphe,White} in dimension three.
The starting point for this paper is the following result of Adams and Reid \cite{AdamsReid} which imposes even stronger restrictions on the systole length for a certain class of hyperbolic manifolds.  Given a closed orientable $3$--manifold $M$, say that a link $L \subset M$ is {\em hyperbolic} if $M \setminus L$ admits a complete hyperbolic metric.  We assume that if $L \subset M$ is a hyperbolic link, then $M \setminus L$ is given this hyperbolic metric.

\begin{thm}  [Adams--Reid]  \label{T:AR} If $M$ is a closed, orientable $3$--manifold that admits no Riemannian metric of negative curvature, then for any hyperbolic link $L \subset M$, the systole length of $M \setminus L$ is bounded above by $7.35534 \ldots $ \ .
\end{thm}

Our goal is to generalize this to all closed orientable $3$--manifolds $M$.  According to the Geometrization Theorem \cite{Perelman2002,Perelman2003,MorganTian2007,MorganTian2008,KleinerLott2008}, we need only consider the case that $M$ is itself a hyperbolic $3$--manifold.  In this setting, one can construct sequences of closed hyperbolic $3$--manifolds $\{M_n\}$ and hyperbolic links $\{L_n \subset M_n\}$ such that the systole length of $M_n \setminus L_n$ tends to infinity (see Section \ref{S:systolebounds}).   Thus, we cannot expect to find a universal bound on systole lengths of hyperbolic link complements in an arbitrary closed $3$--manifold.   In these examples, however, the hyperbolic volume $\vol(M_n)$ tends to infinity, and so one can hope for a bound on systole lengths for hyperbolic link complements in $M$ in terms of $\vol(M)$.  This is precisely what our main theorem states.

To give the most general statement, we declare $\vol(M) = 0$ for any closed $3$--manifold which does {\em not} admit a hyperbolic metric.  Also, we set $C_0 = \frac{\sqrt{3}}{2V_0} \approx 0.853...$, where $V_0$ is the volume of a regular ideal tetrahedron in $\mathbb H^3$.  Then we have

\begin{thm}\label{thm1} If $M$ is a closed, orientable $3$-manifold and $V = \mathrm{vol}(M)$, then for any hyperbolic link $L \subset M$, the systole length of $M \setminus L$ is bounded above by
\[ \log{\left(  \left( \sqrt{2}(C_0V)^{2/3} + 4 \pi^2 \right)^2 + 8  \right) }. \]
\end{thm}

We note that at $V=0$, the expression above is approximately $7.35663 \ldots$, and hence the case of $M$ non-hyperbolic follows from Theorem \ref{T:AR}.  Our goal is thus to prove Theorem \ref{thm1} for hyperbolic $3$--manifolds $M$.

The proof is based on the observation that if $\vol(M \setminus L)$ is considerably larger than $\vol(M)$ for some link $L$ in a closed hyperbolic manifold $M$, then $M \setminus L$ must have a slope length close to $2 \pi$ on the boundary of some maximal embedded cusp neighborhood of an end.  This is a consequence of the $2\pi$-Theorem of Gromov and Thurston (see \cite{BleilerHodgson}) combined with a result of Besson, Courtois, and Gallot \cite{BCG95} (see also Boland, Connell, and Souto \cite{BCS}). An explicit inequality relating the respective volumes and the minimal slope length that we will use was established by Futer, Kalfagianni and Purcell \cite{FKP}, given as Theorem \ref{fkpbound} here. The other key ingredient for our bound is a new upper bound on the systole length of a non-compact, finite volume hyperbolic $3$-manifold in terms of its volume:

\begin{thm}\label{mainsystole}  If $N$ is a non-compact, finite volume hyperbolic $3$-manifold of volume $V$, then the systole of $N$ has length at most
\[ \max\left\{  7.35534... , \log ( 2(C_0V)^{4/3}+ 8 ) \right\} . \] 
\end{thm}

This bound is not sharp, although we believe that $\frac{4}{3} \log{V}$ here is the best known asymptotic bound for the systole in terms of volume $V$.  This bound should be compared with work of Gendulphe \cite{Gendulphe} who provides an upper bound which is sharp for small volume, but is asymptotic to $2 \log(V)$ as $V$ tends to infinity. The constant $\frac{4}{3}$ is of particular interest because it has been shown \cite{BuserSarnak, KSV} that there exist sequences of hyperbolic surfaces whose systole lengths grow like $\frac{4}{3} \log g$, where $g$ is the genus; furthermore, the constant $\frac{4}{3}$ is sharp in the sense that it cannot be improved by looking at principal congruence subgroups of arithmetic Fuchsian groups, which are the main source of examples realizing the $\frac{4}{3}$ bound \cite{Makisumi}. Note that in contrast, there are no known examples of sequences of hyperbolic 3-manifolds whose systoles grow faster than $\frac{2}{3} \log{V}$; see the discussion at the end of Section \ref{S:systolebounds}.

This paper is organized as follows. We first recall some relevant results in Section 2. Section 3 contains the proof of Theorem \ref{mainsystole}, as well as some discussion of asymptotic relationship between systole length and volume. In Section 4, we establish Theorem \ref{thm1} as an application of Theorem \ref{mainsystole}, and prove a Corollary. 

{\bf Acknowledgments.} The authors thank Alan Reid for helpful conversations regarding this work, and in particular for pointing out Corollary \ref{congruence}, and the referee for valuable comments and pointing out a mistake in an earlier draft of this paper. 

\section{Preliminaries} \label{S:preliminaries}

We refer the reader to the text \cite{BenedettiPetronio} for more details on the topics discussed in this section.

The group of orientation-preserving isometries of the upper half-space model for hyperbolic $3$-space $\HH^3$ is identified with the group $\pslc = \slc / \{ \pm I \}$. 
Given an element 
\[ \gamma = \begin{pmatrix}a & b \\ c & d \end{pmatrix} \in \pslc, \]
either $c = 0$, in which case $\infty$ is fixed by the action of $\gamma$ and $\gamma$ acts as a parabolic translation or by scaling along and/or rotating about some vertical axis, or $c \neq 0$, in which case we may study the action of $\gamma$ on $\mathbb{H}^3$ via the \emph{isometric sphere}, $S_\gamma$, of $\gamma$. This is defined to be the (unique) Euclidean hemisphere on which $\gamma$ acts as a Euclidean, as well as a hyperbolic, isometry. It is a standard fact that the center of $S_\gamma$ is $\frac{-d}{c} = \gamma^{-1}(\infty)$, and the radius of $S_\gamma$ is $\frac{1}{|c|}$. Furthermore, it is a simple observation that the isometric sphere $S_{\gamma^{-1}}$ of the inverse has the same radius as $S_\gamma$ and center $\frac{a}{c} = \gamma(\infty)$. The action of $\gamma$ sends $S_\gamma$ to $S_{\gamma^{-1}}$, and the interior (resp. exterior) of $S_\gamma$ to the exterior (resp. interior) of $S_{\gamma^{-1}}$.  For more on isometric spheres, see Marden \cite{OuterCircles}.

A complete Riemannian metric on a $3$--manifold $N$ with constant sectional curvature $-1$ is called a {\em hyperbolic metric} on $N$.  Given a hyperbolic metric on $N$, there is a discrete, torsion free subgroup $\Gamma < \pslc$, unique up to conjugation, so that $N$ is isometric to the space form $\HH^3/\Gamma$ by an orientation preserving isometry.  Let $\Gamma < \pslc$ be such a discrete group so that $N = \HH^3/\Gamma$ has finite volume.

In the upper half-space model, a {\em horoball} is a Euclidean ball tangent to the boundary sphere $S_\infty = \mathbb{C} \cup \{ \infty \}$, where a horoball based at $\infty$ is all the points above a certain fixed height.  A $\Gamma$-equivariant family of pairwise disjoint horoballs is one that is invariant under the action of $\Gamma$.  Such a family always exists and the quotient of such gives a set of pairwise disjoint {\em cusp neighborhoods} $P_1,\ldots,P_k$ of the ends of $N$, one cusp for each conjugacy class of maximal parabolic subgroups of $\Gamma$.   Each cusp neighborhood $P_i$ is homeomorphic to the product of a torus with a ray $T^2 \times [0,\infty)$.  The boundary $\partial P_i \cong T^2$ inherits a Riemannian metric which is Euclidean.

By a {\em slope} $r_i$ on the boundary of a cusp neighborhood $\partial P_i$, we mean the isotopy class $r_i$ of an essential simple closed curve on the torus $\partial P_i$.  A slope $r_i$ on $\partial P_i$ determines (up to inversion), and is determined by, a nontrivial primitive element in the fundamental group of the end with neighborhood $P_i$, so is independent of the particular neighborhood used to describe it.  The \emph{slope length} of a slope $r_i$ on $P_i$ is the minimal length of a representative of the isotopy class in the Euclidean metric described above.   A manifold obtained by removing the interiors of $P_1 \cup \ldots \cup P_k$ and gluing $k$ solid tori so that in the glued manifold $r_i$ bounds a disk in the $i^{th}$ solid torus, is said to be obtained by {\em $(r_1,\ldots,r_k$)--Dehn filling} on $N$.  This manifold depends, up to homeomorphism, on $N$ and $r_1,\ldots,r_k$ and not on any of the other choices involved in the construction.

Starting with any $\Gamma$-equivariant collection of pairwise disjoint horoballs we can equivariantly expand until a point where two horoballs are tangent.  We call the corresponding set of cusps in the quotient $N$ a \emph{maximal set of cusps}.  This is unique if $N$ has one end, but otherwise depends on the choice of disjoint horoballs we start with.   By conjugating $\Gamma$ if necessary, we may base one of the pair of tangent horoballs at $\infty$ and the other at $0$ so that the point of tangency has height $1$.  Such a pair of horoballs will be said to be in \emph{standard position}. It is a standard property of the hyperbolic metric that the volume in the cusp $T^2 \times [1,\infty)$ above height 1 is twice the area of the maximal cusp torus $T^2 \times \{ 1 \}$.

Now suppose we choose one cusp $C$ of $N = \HH^3/\Gamma$ and expand a $\Gamma$-equivariant collection of horoballs for this cusp (this can easily be extended to a maximal set of cusps).  Following Adams \cite{AdamsWaist}, we call the minimal slope length on the resulting maximal cusp neighborhood of $C$ the \emph{waist size} of $C$.    Equivalently, if we assume that two tangent horoballs defining the neighborhood of $C$ are in standard position, the waist size is the minimal translation length on $\mathbb C$ of a parabolic in $\Gamma$ fixing infinity.  The waist size can be different for each cusp of $N$.  In this paper, we will be primarily concerned with manifolds $N$ with all cusp waist sizes at least $2\pi$.


The $2\pi$-Theorem of Gromov and Thurston \cite{BleilerHodgson} states:
\begin{thm}Suppose $N$ is a complete, finite volume, non-compact hyperbolic $3$-manifold. For $1 \leq i \leq n$, let $\{P_i\}_{i=1}^k$ be disjoint horoball neighborhoods of the cusps of $N$, and let $r_i$ be a slope on $\partial P_i$ with length greater than $2\pi$ for each $i = 1,\ldots,k$. Then the closed $3$-manifold obtained by Dehn filling along the slopes $r_i$ is hyperbolic. \end{thm}
This creates a dichotomy between fillings where at least one filling slope is no longer than $2\pi$ and fillings where all slopes are longer than $2\pi$. In their treatment of the first case, Adams and Reid proved the following result, which will also be useful in the present work:
\begin{thm}[\cite{AdamsReid}, Theorem $3.2$]\label{ARbound} If $N = \HH^3/\Gamma$ has a cusp neighborhood with a slope of length $w > 2$, then the systole length of $N$ is at most $\mathrm{Re}\left( 2 \arccosh{\left( \frac{2+w^2i}{2} \right)} \right)$.  Equivalently, $\Gamma$ contains a loxodromic element with translation length at most $\mathrm{Re}\left( 2 \arccosh{\left( \frac{2+w^2i}{2} \right)} \right)$. \end{thm}

From the slope length $w$ Adams and Reid find a loxodromic element which has, in the worst case, trace equal to $2 + w^2i$.  From the trace one can calculate exactly the translation length, which provides their bound.

We will make use of the concept of the \emph{cusp density} of a non-compact, finite volume hyperbolic manifold, defined as the supremum of the possible proportions of the volume of a manifold which can be contained in a union of pairwise disjoint cusps.  In particular, we will appeal to a result of Meyerhoff \cite{Meyerhoff}, applying a result of B\"{o}r\"{o}czky, which states that there is an upper bound for the cusp density of $C_0 = \frac{\sqrt{3}}{2V_0} \approx 0.853...$, where $V_0$ is the volume of a regular ideal tetrahedron. Thus, given a manifold $N$ of volume $V$, the volume in a single cusp can be no more than $C_0V = (0.853\ldots )V$. We denote this quantity by $V_c = C_0V$, and call it the \emph{maximal cusp volume} of $N$.  For more on cusp densities, see Adams \cite{AdamsDensities}.

\section{Systole Bounds} \label{S:systolebounds}

The goal of this section is to establish a bound for the systole of a non-compact, hyperbolic $3$-manifold $N$, with fundamental group $\pi_1(N) \cong \Gamma < \pslc$, in terms of its volume. The argument we give is inspired by Adams' treatment of the two-dimensional case in \cite{AdamsSystole}; the main difference is that instead of horoballs, we consider the possible locations of isometric spheres in order to bound the trace of a loxodromic element of the isometry group. This bound then corresponds to a bound on the translation length of $\gamma$, as is described by the following elementary lemma.

\begin{lem}\label{lengthbound1}A loxodromic element of $\pslc$ with trace of modulus bounded above by $R$ has translation length at most $\log{\left( R^2 + 4 \right) } $.   \end{lem}

\begin{proof}Conjugate so that the element fixes $0, \infty \in \chat$, and thus has the form 
\[ \begin{pmatrix}\lambda & 0 \\ 0 & \lambda^{-1} \end{pmatrix}, \]
corresponding to a transformation $z \mapsto \lambda^2 z$, for $\lambda \in \mathbb{C}$. Then the translation length $L = \log{| \lambda |^2} = 2 \log{|\lambda|}$. Since the trace has modulus $r \leq R$, we have the equation
\[ \lambda + \lambda^{-1} = r e^{i\theta} \]
for some $\theta \in [0,2\pi)$. Solving this using the quadratic equation yields
\[ \lambda = \dfrac{re^{i\theta} \pm \sqrt{r^2 e^{2i\theta} - 4}}{2}, \]
and thus that we have
\[ \left| \lambda \right| \leq \frac{1}{2} \left( r + \sqrt{r^2 + 4} \right). \]
Notice that this inequality is sharp, as it is achieved when the trace is purely imaginary (i.e. $\theta = \frac{\pi}{2}$ or $\frac{3\pi}{2}$). At the cost of slightly weakening the bound, we may replace $r$ with $\sqrt{r^2+ 4}$, and recall that $r \leq R$, to find that the translation length is no larger than
\[ 2 \log{\left( 2\sqrt{R^2+4} \right)} - 2\log{2} = 2\log{\left( \sqrt{R^2+4} \right)} = \log{\left( R^2+4 \right)}. \]
as required. \end{proof}

\noindent {\bf Remark.} We may apply this to the element from Theorem \ref{ARbound}; since the Adams--Reid bound corresponds to an element with trace (at worst) $2 + w^2 i$, this has modulus $\sqrt{w^4 + 4}$. Thus Lemma \ref{lengthbound1} would give the slightly weaker systole bound of $\log{\left( w^4 + 8 \right)}$. When convenient, we will use the weaker bound for the purposes of easier comparison with other bounds.

\noindent {\bf Definition.} Let $AR(w) = \sqrt{w^4+4}$ be the upper bound for the modulus of the trace of a minimal loxodromic element of $\Gamma$ when $\Gamma$ contains a parabolic which translates by distance (no larger than) $w$. As described above, this is the bound given by the combination of the Adams--Reid bound and Lemma \ref{lengthbound1}.

The main result of this section is the following theorem from the introduction.  Recall from Section \ref{S:preliminaries} that $V_c = C_0V$.

\begin{mainsystolethm} If $N$ is a non-compact, finite volume hyperbolic $3$-manifold of volume $V$, then the systole of $N$ has length at most
\[ \max\left\{  7.35534... , \log{ \left( 2 V_c^{4/3}+ 8 \right) } \right\} . \] \end{mainsystolethm}

\begin{proof}  We separate the proof into two cases.

\noindent
{\bf Case 1.} There exists a cusp of $N$ with a waist size of at most $2 \pi$.

\noindent For this case we can appeal directly to the result of Adams--Reid, Theorem \ref{T:AR}, (or more precisely its proof in \cite{AdamsReid}), to see that the systole length is at most $7.35534\ldots$, completing the proof in this case.

\noindent
{\bf Case 2.} Every cusp of $N$ has waist size greater than $2 \pi$.

\noindent
In the light of Lemma \ref{lengthbound1}, it suffices to establish an upper bound for the complex modulus of the trace of a minimal loxodromic element.  As such, we prove:

\begin{prop}\label{traceprop}If $N = \mathbb H^3/\Gamma$ is a non-compact, finite volume hyperbolic $3$-manifold of volume $V$ with a cusp having waist size greater than $2 \pi$, then $\Gamma$ has a loxodromic element with trace of complex modulus no greater than
\[ \sqrt{ 2V_c^{4/3} + 4  }. \] \end{prop}

Assuming this proposition we complete the proof of Theorem \ref{mainsystole}.  By Lemma \ref{lengthbound1}, this bound on the complex modulus of the trace gives a systole length of at most 
\begin{align*} \log{ \left( 2V_c^{4/3}+ 4 + 4 \right)} &=  \log{ \left( 2V_c^{4/3}+ 8 \right)} 
\end{align*}
which is the stated bound. \end{proof}

\noindent {\bf Remark.} In the following, we will often consider only the ``worst case" eventualities in order to obtain an upper bound that is as general as possible. As such, the argument and the resulting bound can likely be improved under more assumptions about $N$ (or $\Gamma$).

\begin{proof}[Proof of Proposition \ref{traceprop}]
We begin by choosing a maximal set of cusps for $N$, so that some cusp has a self-tangency.  After conjugating if necessary we assume, as we may, that this self-tangency lifts to a tangency between a pair of horoballs $H_0,H_\infty$ based at $0$ and $\infty$, respectively, in standard position.  Thus $H_\infty$ and $H_0$ are equivalent under the action of $\Gamma$ and are tangent at height $1$. There is then an element $\gamma \in \Gamma$ such that $\gamma(0) = \infty$ and $\gamma(H_0) = H_\infty$; furthermore, this $\gamma$ has isometric sphere $S_\gamma$ with center $0$ and radius $1$. As such, we have that
\[ \gamma = \begin{pmatrix} a & -\frac{1}{c} \\ c & 0 \end{pmatrix} \in \Gamma < \mathrm{PSL}_2(\mathbb{C}) \]
where $|c| = 1$. Observe that the trace of $\gamma$ is $\tr{\gamma} = a$.  

The element $\gamma \in \Gamma$ is not the unique element taking $0$ to $\infty$ as we can compose with any element of the stabilizer of $\infty$ which we denote $\Gamma_\infty < \Gamma$.  We choose $\gamma$ as above with $|\tr \gamma| = |a|$ minimal among all $\gamma \in \Gamma$ with $\gamma(0) = \infty$.
Based on the value of this trace, the following lemma produces a loxodromic element whose trace we can bound.


\begin{lem}\label{tracelemma}Suppose the Kleinian group $\Gamma$ is torsion-free, that $\gamma \in \Gamma < \pslc$ has the form 
\[ \gamma = \begin{pmatrix} a & -\frac{1}{c} \\ c & 0 \end{pmatrix} \]
with $|c|=1$ and $|a|$ minimal, that $\Gamma_\infty$ has coarea at most $V_c$ when viewed as acting by translations on $\mathbb{C}$ (or equivalently, on the horosphere at height $1$), and that the minimal parabolic translation in $\Gamma_\infty$ is by length $\ell > 2\pi$. Then
\[ 2\pi < \ell \leq \sqrt{\frac{4}{\sqrt{3}} V_c} \]
and either:
\begin{enumerate}[1.]
\item $| \tr{\gamma} | = |a| <2$, in which case $\Gamma$ contains a loxodromic element of trace modulus no greater than $2$; or
\item $| \tr{\gamma} | = |a| = 2$, in which case $\Gamma$ contains a loxodromic element of trace modulus no greater than 
\[ \sqrt{\ell^2+ 4};\] or
\item $| \tr{\gamma} | = |a| > 2$, in which case $\Gamma$ contains a loxodromic element of trace modulus no greater than
\[ \sqrt{ \frac{\ell^2}{4}  + \frac{V_c^2}{\ell^2} }. \]
\end{enumerate}
\end{lem}

\begin{proof}

Since we assume the minimal parabolic translation in $\Gamma_\infty$ is by $\ell$, if necessary we may conjugate $\Gamma$ by a rotation of $\mathbb C$ in $\pslc$ to make this parabolic equal to 
\[ \beta_\ell = \begin{pmatrix}1 & \ell \\ 0 & 1\end{pmatrix}.\]
We extend this to a basis for $\Gamma_\infty \cong \bZ^2$ so that the second generator has the minimal possible translation length.  Since the translation length is at least $\ell$, and since we can compose with powers of $\beta_\ell$ to reduce this, it follows that the second basis element translates by a complex number $z$ in the set
\[ J_\ell := \left\lbrace z \in \mathbb{C} \ \Bigg| \ |z|\geq \ell \mbox{ and } -\frac{\ell}{2} \leq \mathrm{Re}(z) \leq \frac{\ell}{2} \right\rbrace. \]
We find that $z = x + iy$ must have imaginary part $y$ at least $\frac{\sqrt{3}}{2} \ell$. Since the area of the boundary of the cusp is twice the volume of the cusp, and this is at most $V_c$, we have
\[ 2V_c \geq \ell y \geq \frac{\sqrt{3}}{2} \ell^2, \]
and so
\[ \ell \leq \sqrt{\frac{4}{\sqrt{3}} V_c }.\]
This proves the first claim of the lemma.

We now divide the proof up into the three cases in the statement of the lemma.

\noindent
{\bf Case 1.}  $|\tr \gamma | < 2$.

In this case, if the trace is real, then $\gamma$ is an elliptic element. Since $\Gamma$ is discrete, such an element has finite order and $\Gamma$ contains no such elements, hence this cannot occur.  Thus $\tr{\gamma} \notin \mathbb{R}$, and $\gamma$ is a loxodromic element with $|\tr{\gamma}| < 2$. $\qed$

For the next two cases, we make an observation:  Note that the isometric sphere $S_\gamma$ has center $0$ and radius $\frac{1}{|c|} = 1$, while $S_{\gamma^{-1}}$, the isometric sphere of $\gamma^{-1}$, has center $\frac{a}{c}$ and the same radius $\frac{1}{|c|}=1$.  Since $|c|=1$, we see that $ \left| \frac{a}{c} \right| = |a|$. Thus the distance between the centers of $S_\gamma$ and $S_{\gamma^{-1}}$ is equal to $| \tr{\gamma} | = |a| = \left| \frac{a}{c} \right|$.

Consider the Dirichlet fundamental domain $\Delta$ centered at $0$ for the action of $\Gamma_\infty$ on $\mathbb C$ (see Figure \ref{cuspfigure2}, the Dirichlet domain is the dotted region).
Composing $\gamma$ with an element of $\beta \in \Gamma_\infty$ we produce $\alpha = \beta \circ \gamma$ which also has $\alpha(0) = \infty$.  The isometric sphere of $\alpha$ is $S_\alpha = S_\gamma$, while the isometric sphere of $\alpha^{-1}$ is $S_{\alpha^{-1}} = \beta (S_{\gamma^{-1}})$.   Since we have assumed that $\gamma$ has minimal trace, it follows that the center $\frac{a}{c}$ for $S_{\gamma^{-1}}$ lies in $\Delta$.  As the third possibility of the lemma is simpler to handle, we deal with that case next.

\noindent
{\bf Case 3.} $|\tr \gamma| > 2$.

As in case 1, $\gamma$ is the desired loxodromic and so we are left to bound the modulus of its trace, or as just explained, the distance between the centers of the isometric spheres $S_\gamma$ and $S_{\gamma^{-1}}$.   From the discussion above, this is the maximum distance from the center of $\Delta$ to any other point of $\Delta$, or equivalently, the diameter of the quotient $\mathbb C/\Gamma_\infty$.  This diameter is maximized for the rectangular torus with area exactly $2V_c$.  In this case, the diameter is the length of $\frac{1}{2}(\ell + i \frac{2V_c}{\ell})$, which is half the diagonal of the rectangle, or 
\[ \frac{1}{2} \sqrt{\ell^2 + \frac{4V_c^2}{\ell^2}} = \sqrt{ \frac{\ell^2}{4} + \frac{V_c^2}{\ell^2}}. \]
This completes the proof for case 3. \qed


\begin{figure}[htb]
\begin{center}
\includegraphics[scale=0.65]{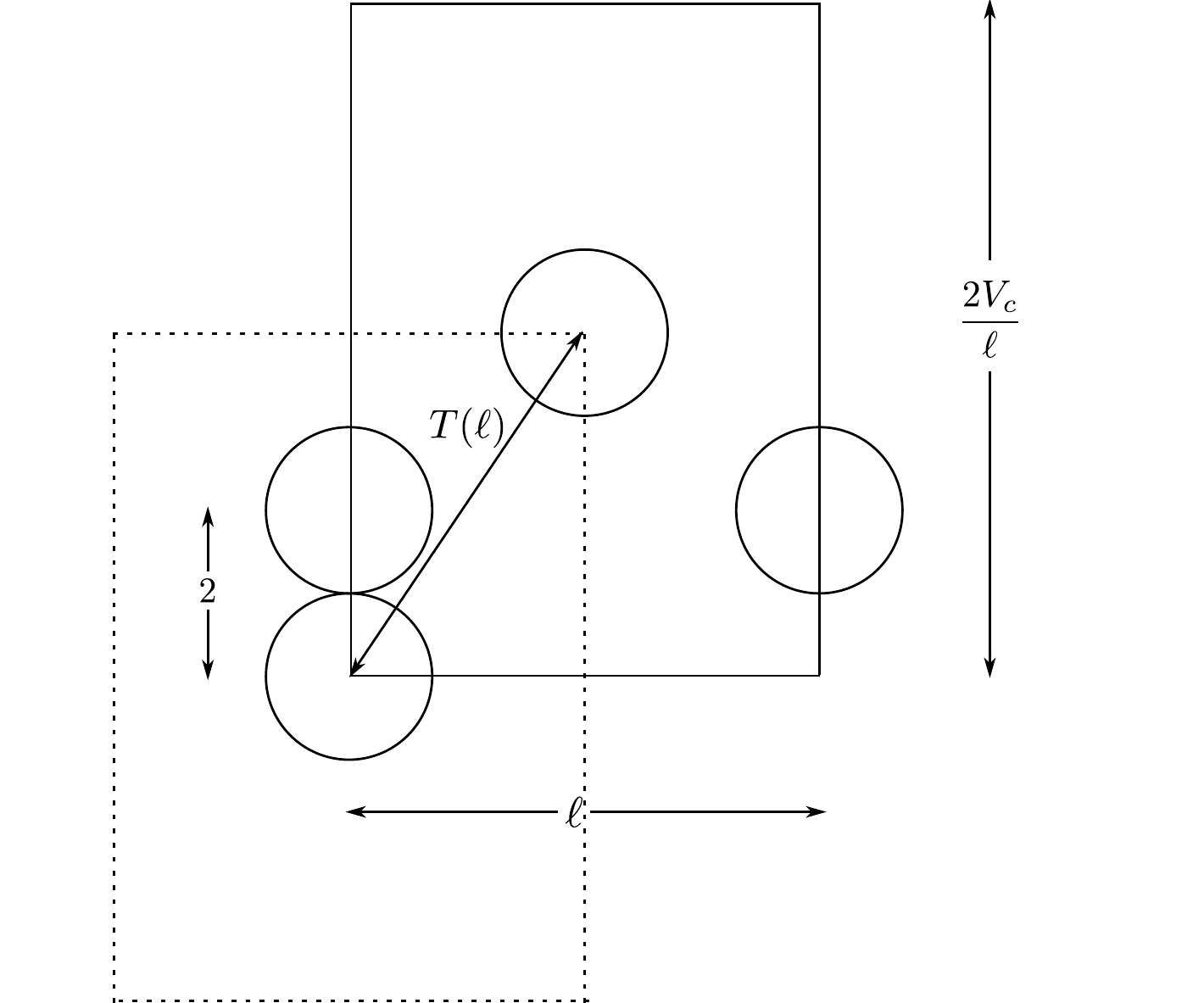}
\caption{\label{cuspfigure2} Trace of the loxodromic is no greater than the larger of $T(\ell)$ and $\sqrt{\ell^2+4}$}
\end{center}
\end{figure}

We now turn to the remaining possibility in the lemma.

\noindent
{\bf Case 2.} $|\tr \gamma| = 2$.

First observe that if $\tr{\gamma} \notin \mathbb{R}$, then $\gamma$ is a loxodromic element and we are done.  So, we assume that $\tr{\gamma} \in \mathbb{R}$, so that $\gamma$ is a parabolic element, fixing the point in $\mathbb{C}$ of tangency of $S_\gamma$ and $S_{\gamma^{-1}}$. 
Composing $\gamma$ with an element $\beta \in \Gamma_\infty$ produces an element $\alpha = \beta \circ \gamma$ with $S_\alpha = S_\gamma$ and $S_{\alpha^{-1}} = \beta ( S_{\gamma^{-1}} )$.  The required loxodromic element will be such an element $\alpha$.

Since $|\tr(\alpha)|$ is the distance from $0$ to the center of $S_{\alpha^{-1}} = \beta (S_{\gamma^{-1}})$, we seek to bound the distance of a $\beta$--translate of the center of $S_{\gamma^{-1}}$, for translates such that the center of $S_{\gamma^{-1}}$ is further than distance $2$ from $0$, and so $\alpha$ is loxodromic.
We are guaranteed that
\[ \beta_\ell^{\pm 1} \circ \gamma = \begin{pmatrix} 1 & \pm \ell \\ 0 & 1 \end{pmatrix} \begin{pmatrix} a & \frac{-1}{c} \\ c & 0 \end{pmatrix} = \begin{pmatrix} a \pm c\ell & \frac{-1}{c} \\ c & 0 \end{pmatrix}, \]
are both loxodromic, since both $\beta_\ell^{\pm}$ translate a distance $\ell > 2 \pi$.
Therefore we want to bound the minimum of $|a + c \ell|$ and $|a - c \ell|$.

The locus of possible centers $\frac{a}{c}$ forms a circle $\Omega$ of radius $2$ (see Figure \ref{parabolicfigure}).  Write $\Omega = \Omega_+ \cup \Omega_-$, where $\Omega_+$ are those points with nonnegative real part and $\Omega_-$ are those with nonpositive real parts.   Now for $\frac{a}{c} \in \Omega_\pm$ we have
\[ |\tr (\beta_\ell^{\mp 1} \circ \gamma) | = |a \mp c \ell| \leq \sqrt{\ell^2 + 4};\]
see Figure \ref{parabolicfigure}.
This completes the proof in this case.


\begin{figure}[htb]
\begin{center}
\includegraphics[scale=0.65]{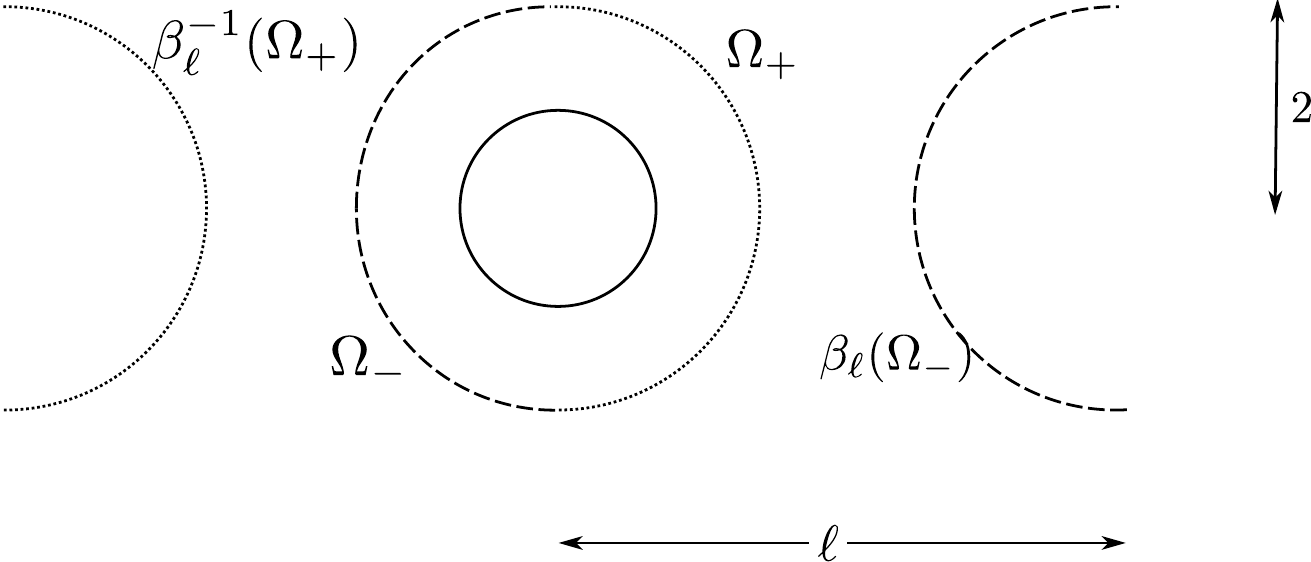}
\caption{\label{parabolicfigure} $\Gamma$ contains a loxodromic element of trace modulus at most $\sqrt{\ell^2+4}$.}
\end{center}
\end{figure}


\end{proof}

\noindent {\bf Definition.} We will denote by $T(\ell)$ the function 
\[T(\ell) = \sqrt{ \frac{\ell^2}{4}+ \frac{V_c^2}{\ell^2}  }.\]

In the event that $\ell$ is close to $2\pi$, it may be that the loxodromic element of trace $2 + \ell^2 i$, given by Adams--Reid, has smaller trace modulus $AR(\ell)$ than the loxodromic element given by case 3. As such, we may replace the function $T(\ell)$ in case 3 by the minimum $S(\ell) =\min{\left( T(\ell), AR(\ell)  \right) }$.

In the light of Lemma \ref{tracelemma}, and because we assume $\ell > 2\pi$ (so that the bounds from cases 2 and 3 exceed the trace bound of 2 given by case 1), we seek to bound the maximum of $S(\ell)$ and $\sqrt{\ell^2+4}$ for $\ell \in \left\lbrack 2\pi , \sqrt{\frac{4}{\sqrt{3}} V_c} \right\rbrack$. The function $\sqrt{\ell^2+4}$ is increasing, and thus is bounded above by its value at the endpoint $\ell = \sqrt{\frac{4}{\sqrt{3}}V_c}$, which is $\sqrt{\frac{4}{\sqrt{3}}V_c + 4}$. A bound for $S(\ell)$ is given by the following Lemma.

\begin{lem}\label{techlem2} For $2\pi \leq \ell \leq \sqrt{\frac{4}{\sqrt{3}} V_c}$, the function $S(\ell) = \min{\left( T(\ell), AR(\ell) \right)}$ is bounded above by 
\[ AR(2^{1/4} V_c^{1/3}) = \sqrt{2V_c^{4/3}+4}.\] 
\end{lem}
\begin{proof}
To simplify the calculation here, square both $AR$ and $T$, multiply them by $4$, and precompose by the square root function.   This produces two new functions $a(x) = 4 x^2 + 16$ and $t(x) = x + \frac{4V_c^2}{x}$.  To prove the lemma it suffices to bound the minimum $s(x) = \min\{ a(x),t(x) \}$ on the interval $\left( 0,\frac{4 V_c}{\sqrt{3}}\right]$ by $a(\sqrt{2} V_c^{2/3}) = 8V_c^{4/3} + 16$.

First note that the cubic function $f(x) = x(a(x) - t(x)) = 4x^3 - x^2 + 16x - 4V_c^2$ is monotone since $f'(x) = 12x^2 - 2x + 16 > 0$ for all $x$.  Since $f(0) = - 4V_c^2 < 0$ it follows that $a(x) < t(x)$ for all sufficiently small $x > 0$.  Furthermore, the leading coefficient of $f$ is $4> 0$ so $f(x) = 0$ has a unique solution $x_0 > 0$,  and hence $a(x) = t(x)$ for exactly one $x = x_0 > 0$.  Consequently, for $x > 0$, $a(x) > t(x)$ if and only if $x > x_0$.

\noindent
{\bf Claim.} $x_0 < \sqrt{2} V_c^{2/3}$.

It suffices to show that $a(\sqrt{2} V_c^{2/3}) - t(\sqrt{2} V_c^{2/3}) > 0$.  Plugging in we get
\[ a(\sqrt{2} V_c^{2/3})-t(\sqrt{2} V_c^{2/3}) = 8V_c^{4/3} + 16 - \sqrt{2} V_c^{2/3} - \frac{4V_c^2}{\sqrt{2}V_c^{2/3}} = (8-2\sqrt{2}) V_c^{4/3} - \sqrt{2}V_c^{2/3} + 16.\]
Considering the discriminant of the quadratic $(8-2\sqrt{2}) z^2 - \sqrt{2}z + 16$, we see that this is positive for every real number $V_c$, and the claim follows. \qed

Now we see that $s(x)$ is equal to $a(x)$ on $(0,x_0]$ and $t(x)$ on $\left[ x_0, \frac{4 V_c}{\sqrt{3}} \right]$.  Since $a(x)$ is increasing on $[0,\infty)$, $s(x)$ is bounded above by $a(x_0) < a(\sqrt{2} V_c^{2/3})$ on the interval $(0,x_0]$.  

Now we prove that this bound for $s(x)$ is also valid on $\left[ x_0, \frac{4 V_c}{\sqrt{3}} \right]$.  Calculating $t'(x) = 1 - \frac{4V_c^2}{x^2}$ and $t''(x) = \frac{8V_c^2}{x^3}$ we see that on the interval $\left[ x_0, \frac{4 V_c}{\sqrt{3}} \right]$, $t(x)$ is convex and has a unique minimum at $x = 2V_c$ (here we're using the fact that $V_c > 1$ to guarantee $V_c > V_c^{2/3}$).   Since $t(x_0) = a(x_0) < a(\sqrt{2} V_c^{2/3})$, it follows that for $x \in \left[ x_0, \frac{4 V_c}{\sqrt{3}} \right]$, we have
\[ t(x) \leq \max \left\{ t(x_0),t\left(\frac{4 V_c}{\sqrt{3}}\right) \right\} \leq \max \left\{a( \sqrt{2} V_c^{2/3}),t\left(\frac{4 V_c}{\sqrt{3}}\right) \right\} .\]
Since
\[ t\left( \frac{4V_c}{\sqrt{3}} \right) = \frac{4V_c}{\sqrt{3}} + \frac{4V_c^2\sqrt{3}}{4V_c} = \frac{7 V_c}{\sqrt{3}} < 8V_c^{4/3} + 16 = a(\sqrt{2} V_c^{2/3}) \]
it follows that $t(x) < a(\sqrt{2} V_c^{2/3})$ as required.
\end{proof}

It remains to check when this quantity also bounds the maximum of $\sqrt{\ell^2+4}$ for $\ell \in \left[ 2 \pi, \sqrt{\frac{4 V_c}{ \sqrt{3} }} \right]$. This is true when 
\[ \sqrt{2V_c^{4/3}   +4  } \geq \sqrt{\frac{4}{\sqrt{3}}V_c + 4}, \]
and so when
\[ 2V_c^{4/3} \geq \frac{4V_c}{\sqrt{3}}. \]
As this only requires $V_c \geq \frac{8}{3 \sqrt{3}} \approx 1.5396\ldots$ and the volume of a maximal cusp with waist size at least $2 \pi$ is at least
\[ \frac{\sqrt{3}}{4}(2\pi)^2 \approx 17.094... \ \ , \]
it follows that for all manifolds $M$ satisfying this condition, the bound given by Lemma \ref{techlem2} is a bound for the modulus of the trace of a loxodromic element.   This completes the proof of the proposition.
\end{proof}

Due to the estimates involved in the above argument, this bound is not sharp.
Furthermore, we do not know of a sequence of cusped hyperbolic $3$-manifolds which have the property that the systole length grows in proportion with $\log{\left( V^{4/3} \right)}$. The following is an example where the systole length is proportional to $\log{\left( V^{2/3} \right)}$; this compares with the work of Katz, Schaps and Vishne \cite{KSV}, who construct examples of compact hyperbolic $3$-manifolds where the systole length also grows like $\log{\left( || M ||^{2/3} \right)}$, where $||M||$ denotes the simplicial volume of $M$.

\noindent {\bf Example.} Consider the Bianchi group $\mathrm{PSL}_2(\mathcal{O}_2)$, where $\mathcal{O}_2$ denotes the ring of integers $\mathbb{Z}[\sqrt{-2}]$ of the imaginary quadratic number field $\mathbb{Q}(\sqrt{-2})$. An element $a+b\sqrt{-2} \in \mathcal{O}_2$, with $a, b \in \mathbb{Z}$, has norm 
\[ N(a+b\sqrt{-2}) = a^2 + 2b^2.\]
The prime $11$ splits in $\mathcal{O}_2$ as $11 = (3 + \sqrt{-2})(3-\sqrt{-2})$. Let $\pi = 3 + \sqrt{-2}$, so $N(\pi) = 11$, and consider the sequence of principal congruence subgroups $\Gamma(\pi^n) <  \mathrm{PSL}_2(\mathcal{O}_2)$, where
\[ \Gamma(N) = \mathrm{P} \left\lbrace \begin{pmatrix}a & b \\ c & d \end{pmatrix} \in \mathrm{SL}_2(\mathcal{O}_2) \ \Big| \  a \equiv d \equiv 1, b \equiv c \equiv 0 \mbox{ mod }  N  \right\rbrace . \]
These groups are all torsion-free and so the quotients $\mathbb{H}^3/\Gamma(\pi^n)$ are hyperbolic $3$-manifolds. The volume $V_n$ of each is the product of the volume of the base orbifold $\mathbb{H}^3/\mathrm{PSL}_2(\mathcal{O}_2)$ multiplied by the index $\left[ \mathrm{PSL}_2(\mathcal{O}_2) : \Gamma(\pi^n) \right]$; a formula for this index is given by Newman \cite{Newman}:
\[  \left[ \mathrm{PSL}_2(\mathcal{O}_2) : \Gamma(\pi^n) \right] = \frac{(N(\pi^n))^3}{2} \left( 1 - \frac{1}{N(\pi)^2} \right) = \frac{11^{3n}}{2} \frac{120}{121};   \]
thus, the volume is $V_n = 11^{3n} C$ for some constant $C >  0$. Since $\mathcal{O}_2$ is a principal ideal domain, a generic element of $\Gamma(\pi^n)$ has the form
\[ \begin{pmatrix} a \pi^n +1 & b \pi^n \\ c \pi^n & d \pi^n +1 \end{pmatrix}, \]
where $a, b, c, d \in \mathcal{O}_2$. Thus the trace has the form $(a+d) \pi^n + 2$, and the determinant gives the equation
\[ ad\pi^{2n} + (a+d)\pi^n + 1 - bc \pi^{2n} = 1, \]
and so
\[ \pi^n \left( ad \pi^n - bc \pi^n + (a+d) \right) = 0. \]
Therefore, we see that $(a+d) \in (\pi^n)$. If $(a+d) =0 $ then the trace is $2$, corresponding to a parabolic element. Thus to find the minimal possible loxodromic trace, we may assume $(a+d) \neq 0$. Because $\mathcal{O}_2$ contains no nonzero elements of modulus less than $1$, the complex modulus $|a+d|$ must be at least that of $\pi^n$, which is equal to $11^{n/2}$, because $|\pi| = \sqrt{11}$. Thus the minimal loxodromic trace $(a+d)\pi^n +2$ has complex modulus at least $| \pi^{2n} | = 11^n$. Recalling the proof of Lemma \ref{lengthbound1}, and if necessary replacing $\lambda$ with $\lambda^{-1}$, we have that if a loxodromic isometry has trace at least $R$, then $| \lambda | \geq \frac{R}{2}$, and thus the systole length is at least $\log{\left( \frac{R^2}{4} \right)}$. Therefore, the present case gives a systole length $l_n$ of at least $\log{ \left( \frac{11^{2n}}{4} \right)} = 2n\log{11} - \log{4}$. Taking the logarithm of the volume found above yields
\begin{align*} \log{V_n} &= 3n \log{11} + \log{C} \\
                                         &\leq \frac{3}{2} l_n + \frac{3}{2} \log{4} + \log{C}, \end{align*}
and so
\[ l_n \geq \frac{2}{3} \log{V_n} - \log{4} - \frac{2}{3}\log{C}. \]

Given Theorem \ref{mainsystole}, it is natural to ask whether one can find a sequence of hyperbolic 3-manifolds whose systole lengths grow faster than these examples. That is:

\noindent {\bf Question.} Do there exist hyperbolic manifolds $\{ M_n \}$ such that $V_n = \mathrm{vol}(M_n) \to \infty$ and the systole lengths $\mathrm{sys}(M_n)$ grow like $\log{V_n^\varepsilon}$ for $\varepsilon \in (\frac{2}{3}, \frac{4}{3} ]$ ?

\section{Links in Manifolds}

In this section, we prove the following theorem.

\begin{linkthm}\label{main}Let $M$ be a compact hyperbolic $3$-manifold of volume $V$, and let $L \subset M$ be a link such that the complement $M \setminus L$ is hyperbolic. Then the systole length of $M \setminus L$ is bounded above by
\[ \log{\left( 2 \left( ( (0.853\ldots) V)^{2/3} + 4 \pi^2 \right)^2 + 8  \right) }. \]\end{linkthm}
As in previous sections, we write $C_0 = 0.853...$ to make the expressions easier to read.

\begin{proof} Write $M \setminus L = \mathbb H^3/\Gamma_L$.  We will consider two functions which bound the modulus of a loxodromic element of $\Gamma_L$ in terms of the volume of $M \setminus L$, which we denote $X = \mathrm{vol}(M \setminus L)$.  The first is the bound established in Theorem \ref{mainsystole}
\[ F_1(X) = \sqrt{ 2(C_0X)^{4/3}  + 4}. \]

The second function comes from the following theorem of Futer, Kalfagianni and Purcell \cite{FKP}:

\begin{thm}[\cite{FKP}, Theorem 1.1]\label{fkpbound} Let $M \setminus L$ be a complete, finite volume, noncompact hyperbolic manifold. Let $M$ denote the closed manifold obtained by filling along specified slopes on each of the boundary tori, each of which have length at least $2\pi$, and the least of which is denoted by $\ell_{\text{min}}$. Then $M$ is hyperbolic, and 
\[ \mathrm{Vol}(M) \geq \left( 1 - \left( \frac{2\pi}{\ell_{\text{min}}} \right)^2 \right)^{3/2} \mathrm{Vol}(M \setminus L).\] \end{thm}

Solving the inequality given in the theorem for the minimal parabolic translation distance $\ell_{\mathrm{min}}$ we find that
\[ \ell_{\mathrm{min}} \leq \dfrac{2\pi}{\sqrt{ 1- \left( \dfrac{V}{\mathrm{Vol}(M \setminus L)} \right)^{2/3} }  } = \dfrac{2\pi}{\sqrt{ 1- \left( \dfrac{V}{X} \right)^{2/3} }  } . \]
Thus $\ell_{\mathrm{min}}$ is bounded above by a function of $X$ with $X> V$.  
Applying the combination of Theorem \ref{ARbound} and Lemma \ref{lengthbound1}, as described in the remark following Lemma \ref{lengthbound1}, we can bound the modulus of the trace of a loxodromic element of $\Gamma_L$ by $AR(\ell_{\mathrm{min}})$.  Since $AR$ is increasing, the modulus of the trace is thus bounded above by the following, which we also view as a function of $X$.
\[ F_2(X) =  \sqrt{ \dfrac{16 \pi^4}{\left( 1- \left( \dfrac{V}{X} \right)^{2/3} \right)^2  }  + 4 }. \]
The modulus of the trace of a loxodromic element of $\Gamma_L$ is thus bounded above by the minimum of $F_1(X)$ and $F_2(X)$.

Now observe that $F_1(X)$ is an increasing function and $F_2(X)$ is a decreasing function.  
Consequently, if there is a common value of these two functions, then it occurs at a unique point, and furthermore, the value at the unique point where they agree serves as an upper bound for the modulus of the trace of a loxodromic element of $\Gamma_L$.



Setting $F_1$ and $F_2$ equal, we find
\[  \dfrac{16 \pi^4}{\left( 1- \left( \dfrac{V}{X} \right)^{2/3} \right)^2} =  2(C_0X)^{4/3} \]
which, with some rearrangement, becomes
\begin{align*} \frac{16 \pi^4}{2C_0^{4/3}} &= X^{4/3}\left( 1- \left( \dfrac{V}{X} \right)^{2/3} \right)^2 \\
&= \left( X^{2/3} - X^{2/3} \left( \dfrac{V}{X} \right)^{2/3} \right)^2 \\
&= \left( X^{2/3} - V^{2/3} \right)^2. \end{align*}
Taking square roots, and recalling that $X > V$, we obtain
\[ \frac{4 \pi^2}{\sqrt{2}C_0^{2/3}} =  X^{2/3} - V^{2/3} \]
and so
\[ X = \left( V^{2/3} +  \frac{4 \pi^2}{\sqrt{2}C_0^{2/3}} \right)^{3/2}. \]
Plugging this value for $X$ into $F_1$, and then applying Theorem \ref{mainsystole}, gives as a bound for the systole
\begin{align*}
\log(F_1(X)^2 + 4) & =  \log\left( 2 \left( C_0 \left( V^{2/3} +  \frac{4 \pi^2}{\sqrt{2}C_0^{2/3}} \right)^{3/2} \right)^{4/3} + 4 + 4 \right)\\
& =  \log\left(  \left( \sqrt{2}(C_0V)^{2/3} + 4 \pi^2 \right)^2 + 8 \right)
\end{align*}
as required.\end{proof}

\noindent {\bf Example.} We recall the examples from the end of Section \ref{S:systolebounds} setting $N_n = \HH^3/\Gamma(\pi^n)$.  This has volume $V_n \to \infty$ as $n \to \infty$ and systole length $l_n$ satisfying 
\[ l_n \geq \frac{2}{3} \log{V_n} - \log{4} - \frac{2}{3}\log{C}\]
for some constant $C>0$.  By Thurston's Dehn surgery theorem \cite{BenedettiPetronio}, we see that for all but finitely many Dehn fillings on $N_n$, the volume of any of the resulting manifolds $M_n$ is at least $V_n-1$.  On the other hand, each such $M_n$ contains a link $L_n \subset M_n$ (the union of the core curves of the filling solid tori), so that $M_n \setminus L_n \cong N_n$.

These examples illustrate the fact that there is no universal bound on the systole length of a hyperbolic link in a closed hyperbolic manifold since $M_n \setminus L_n$ has systole length $l_n \to \infty$.  Moreover, this shows that one cannot do better than a logarithmic bound on systole length in terms of volume of the manifold.  As with the question at the end of Section \ref{S:systolebounds}, it is natural to ask if one can asymptotically improve the coefficient $4/3$ of the logarithm in Theorem \ref{thm1} to some number between $2/3$ and $4/3$.\\

We end the paper with an application of Theorem \ref{thm1}.

\begin{cor}\label{congruence} For a fixed compact hyperbolic $3$-manifold $M$, there are only finitely many (homeomorphism types of) hyperbolic link complements $M \setminus L$ such that $\pi_1( M \setminus L)$ is isomorphic to a principal congruence subgroup $\Gamma(I)$ of some Bianchi group $\mathrm{PSL}_2(\mathcal{O}_d)$.\end{cor}

\begin{proof}
Recall that for any link $L$, $M \setminus L$ is the interior of a compact manifold $\bar N$ with boundary $\partial \bar N$ a disjoint union of $n$ tori, where $n = |L|$, the number of components of $L$.  The {\em cuspidal cohomology} of $\bar N$ is the image of the map $H^*(\bar N,\partial \bar N;\bQ) \to H^*(\bar N;\bQ)$ from the long exact sequence of the pair $(\bar N,\partial \bar N)$.  From the long exact sequence, Poincar\'e-Lefschetz duality, and Euler characteristic considerations, the dimension of the degree $1$ cuspidal cohomology is equal to $dim(H^1(\bar N ; \bQ)) - n = dim(H_2(\bar N,\partial \bar N;\bQ)) - n$.

On the other hand, we can consider the long exact sequence in homology of the pair $(M,L)$.  From this and excision, we have $dim(H_2(\bar N, \partial \bar N; \bQ)) = dim(H_2(M,L;\bQ))  \leq dim(H_2(M;\bQ)) + n$.  Therefore, the dimension of the degree $1$ cuspidal cohomology of $\bar N$ is at most $dim(H_2(M;\bQ))$.

Now suppose there exists a sequence of links $L_k \subset M$ so that $M \setminus L_k \cong \HH^3/\Gamma_k$, where $\Gamma(I_k) < \psl_2(\mathcal O_{d_k})$, is an infinite sequence of congruence subgroups of Bianchi groups.  According to a result of Grunewald-Schwermer \cite{GS}, because the dimension of the degree $1$ cuspidal cohomology is bounded, the set $\{d_k\}_{k=1}^\infty$ is a finite set.  After passing to a subsequence, we may assume $d_k = d$, for all $k$.  Therefore, $\Gamma(I_k) < \psl_2(\mathcal O_d)$ for all $k$.

Theorem \ref{thm1} gives an upper bound $t_M$ for the systole length of $M \setminus L_k$, depending on the volume $V = vol(M)$. This in turn gives an upper bound on the complex modulus of the trace of a minimal loxodromic element of $\pi_1(M \setminus L_k)$. This trace is equal to $2+x_k$ for some $0 \neq x \in I_k$, and so there is an upper bound on the complex modulus of $x_k$. Since $\mathcal{O}_k$ is a Dedekind domain, there are only finitely many ideals of bounded norm in $\mathcal{O}_d$; hence, there are only finitely many ideals of $\mathcal{O}_k$ which contain an element $x_k$ of bounded norm, and hence of bounded complex modulus. Therefore, only finitely many principal congruence subgroups $\Gamma(I_k) < \mathcal{O}_d$ possess a systole no longer than $t_M$.
\end{proof}

\noindent
{\bf Remark.} In fact, one may conclude that for any $V> 0$, among all closed orientable manifolds with $\vol(M) < V$ there are at most finitely many link complements $M \setminus L$ which are homeomorphic to the quotient of $\mathbb H^3$ by a principal congruence subgroup of a Bianchi group.

\bibliographystyle{amsplain}
\bibliography{systolerefs}

\noindent Department of Mathematics,\\
University of Illinois at Urbana-Champaign,\\
1409 W. Green St,\\
Urbana, IL 61801.\\
Email: lakeland@illinois.edu

\noindent Department of Mathematics,\\
University of Illinois at Urbana-Champaign,\\
1409 W. Green St,\\
Urbana, IL 61801.\\
Email: clein@math.uiuc.edu

\end{document}